\documentclass[letter]{amsart}
\usepackage{amsmath,amsfonts,amsthm,amssymb}
\usepackage{hyperref}

\theoremstyle{plain}
\newtheorem{theorem}{Theorem}[section]
\newtheorem{hypothesis}{Hypothesis}

\newtheorem{lemma}[theorem]{Lemma}
\newtheorem*{de-lemma}{Lemma}

\newtheorem{corollary}[theorem]{Corollary}

\theoremstyle{remark}
\newtheorem*{remark}{Remark}
\newtheorem*{remarks}{Remarks}

\theoremstyle{definition}

\DeclareMathOperator{\dv}{div}
\DeclareMathOperator{\sgn}{sgn}

\newcommand{\dd}{\mathrm{d}}
\newcommand{\R}{\mathbb{R}}
\newcommand{\SF}{\mathbb{S}}
\newcommand{\n}{\textbf{\em n}}

\begin{document}

\title[A new proof for the existence of an equivariant entire solution]{A new proof for the existence of an equivariant entire solution connecting the minima of the potential for the system $\Delta u - W_u (u) = 0$}
\author{Nicholas D.\ Alikakos} 
\address{Department of Mathematics\\ University of Athens\\ Panepistemiopolis\\ 15784 Athens\\ Greece \and Institute for Applied and Computational Mathematics\\ Foundation of Research and Technology -- Hellas\\ 71110 Heraklion\\ Crete\\ Greece} 
\email{\href{mailto:nalikako@math.uoa.gr}{\texttt{nalikako@math.uoa.gr}}}
\thanks{The author was supported by Kapodistrias grant No.\ 15/4/5622 at the University of Athens.}

\begin{abstract}
Recently, Giorgio Fusco and the author in \cite{alikakos-fusco} studied the system $\Delta u - W_u (u) = 0$ for a class of potentials that possess several global minima and are invariant under a general finite reflection group, and established existence of equivariant solutions connecting the minima in certain directions at infinity, together with an estimate. In this paper a new proof is given which, in particular, avoids the introduction of a pointwise constraint in the minimization process.
\end{abstract}

\maketitle

\section{Introduction}
The study of the system
\begin{equation}\label{system}
\Delta u - W_u(u) = 0, \text{ for } u: \R^n \to \R^n,
\end{equation}
where $W: \R^n \to \R$ and $W_u := (\partial W / \partial u_1, \dots, \partial W / \partial u_n)^{\top}$, under symmetry hypotheses on the potential $W$ was initiated in Bronsard, Gui, and Schatzman \cite{bronsard-gui-schatzman}, where existence for the case $n=2$ with the symmetries of the equilateral triangle was settled. About twelve years later this work was followed by Gui and Schatzman \cite{gui-schatzman}, where the case $n=3$ for the symmetry group of tetrahedron was established. The corresponding solutions are known as the {\em triple junction} and the {\em quadruple junction} respectively. This class of solutions is characterized by the fact that they {\em connect} the $N$ global minima of the potential $W$, that is,
\begin{equation}\label{connect}
\lim_{\lambda\to+\infty} u(\lambda\eta_i) = a_i, \text{ for } i=1,\dots,N,
\end{equation}
for certain unit vectors $\eta_i\in\SF^{n-1}$, where $\SF^{n-1}\subset\R^n$ is the unit sphere. These solutions are related to minimal surface complexes, and particularly to the singular points there (see Taylor \cite{taylor}, Dierkes {\em et al.} \cite{dhkw1,dhkw2}) via the blow-down limit $u_\varepsilon(x) := u(x/\varepsilon)$ (see Baldo \cite{baldo}). Recently in \cite{alikakos-fusco} certain general hypotheses on $W$ were identified and the problem was settled for general dimension $n$ and for any reflection group $G$ on $\R^n$.

In this paper we want to give a new derivation of this result, which is based on a positivity property of the gradient flow associated to \eqref{system} and comparison arguments involving subharmonic functions, ingredients already existing in \cite{alikakos-fusco}, but now supplemented with a Kato-type inequality and the De Giorgi oscillation lemma. The present paper is self-contained. Our hope is that this simpler proof will be more adaptable to the general case of a potential $W$ without symmetry requirements. In order to bring out clearly the underlying ideas, we refrain from any generalization which could complicate the technical part.

\subsection{Notation}
We denote by $B_R$ the ball of radius $R>0$ centered at the origin and by $W_{\mathrm{E}}^{1,2}(B_R; \R^n)$ the subspace of {\em equivariant maps}, that is, $u(gx) = g u(x)$, for all $g \in G$ and $x \in \R^n$.  We also denote by $\langle\cdot ,\cdot \rangle$ the Euclidean inner product, by $| \cdot |$ the Euclidean norm, and by $d(x, \partial D)$ the distance of $x$ from $\partial D$. In the case of finite groups $G$, the notation $|G|$ stand for the number of elements of the group.

We denote the functional associated to \eqref{system} by
\begin{equation}\label{action}
J(u) = \int_{\R^n} \left\{ \frac{1}{2} |\nabla u|^2 + W(u) \right\} \dd x.
\end{equation}

A {\em Coxeter group} is a finite subgroup of the orthogonal group $O(\R^n)$, generated by a set of reflections. A {\em reflection} $\gamma \in G$ is associated to the hyperplane 
\[ \pi_\gamma = \{ x\in\R^n \mid \langle x, \eta_\gamma \rangle = 0 \}, \]
via 
\[ \gamma x = x - 2 \langle x, \eta_\gamma \rangle \eta_\gamma, \text{ for } x \in \R^n, \]
where $\eta_\gamma \in \SF^{n-1}$ is a unit vector. Every finite subgroup of $O(\R^n)$ has a {\em fundamental region}\footnote{See \cite{grove-benson} or \cite{humphreys}.}, that is, a subset $F \subset \R^n$ with the following properties:
\begin{enumerate}
\item $F$ is open and convex,
\item $F \cap gF = \varnothing$, for $I \neq g \in G$, where $I$ is the identity,
\item $\R^n = \cup\{ g \overline{F} \mid g \in G \}$.
\end{enumerate}
We choose the orientation of $\eta_\gamma$ so that $F \subset \mathcal{P}_{\gamma}^{+}$, where $\mathcal{P}_{\gamma}^{+} = \{ x\in\R^n \mid \langle x, \eta_\gamma \rangle > 0 \}$. Then, we have
\begin{equation}
F = \cap_{\gamma \in \Gamma} \mathcal{P}_{\gamma}^{+},
\end{equation}
where $\Gamma \subset G$ is the set of all reflections in $G$. Given $a \in G$, the {\em stabilizer} of $a$, denoted by $G_a$, is the subgroup of $G$ that fixes $a$.

\subsection{The theorem (\cite{alikakos-fusco})}
We begin with the hypotheses.
\begin{hypothesis}[$N$ nondegenerate global minima]\label{h1}
The potential $W$ is of class $C^2$ and satisfies $W(a_i)=0$, for $i=1,\ldots,N$, and $W>0$ on $\R^n \setminus \{a_1,\dots a_N\}$. Furthermore, there holds $v^\top\partial^2 W(u)v \geq 2 c^2 |v|^2$, for $v\in\R^n$ and $|u-a_i| \leq \bar{q}$, for some $c$, $\bar{q} > 0$, and for $i=1,\ldots, N$.
\end{hypothesis}

\begin{hypothesis}[Symmetry]\label{h2}
The potential $W$ is invariant under a finite reflection group $G$ acting on $\R^n$ (Coxeter group), that is,
\begin{equation}\label{g-invariance}
W(gu) = W(u), \text{ for all } g \in G \text{ and } u \in \R^n.
\end{equation}
Moreover, we assume that there exists $M>0$ such that
$W(su) \geq W(u)$, for $s\geq 1$ and $|u|=M.$
\end{hypothesis}
We seek {\em equivariant} solutions of system \eqref{system}, that is, solutions satisfying
\begin{equation}\label{equivariance}
u(gx) = gu(x), \text{ for all } g \in G \text{ and } x \in \R^n.
\end{equation}

\begin{hypothesis}[Location and number of global minima]\label{h3}
Let $F \subset \R^n$ be a fundamental region of $G$. We assume that $\overline{F}$ (the closure of $F$) contains a single global minimum of $W,$ say $a_1$, and let $G_{a_1}$ be the subgroup of $G$ that leaves $a_1$ fixed. Then, as it follows by the invariance of $W$, the number of the minima
of $W$ is
\begin{equation}\label{enne}
N= \frac{|G|}{|G_{a_1}|}.
\end{equation}
\end{hypothesis}

\begin{hypothesis}[$Q$-monotonicity]\label{h4}
We restrict ourselves to potentials $W$ for which there is a continuous function $Q: \R^n \to \R$ that satisfies
\begin{equation}\label{q-smooth}
Q(u+a_1) = |u| + H(u),
\end{equation}
where $H:\R^n\to\R$ is a $C^2$ function such that $H(0) = 0$ and $H_u(0) = 0$, and
\begin{subequations}\label{q-list}
\begin{align}
&Q \text{ is convex,}\label{q-list-a}\\
&Q(g u)=Q(u), \text{ for } u\in \R^n,\ g\in G_{a_1},\label{q-list-b}\\
&Q(u + a_1) = |u| + H(u), \text{ in a neighborhood of } u=0,\label{q-list-c}\\
&Q(u) >0, \text{ on } \R^n \setminus \{ a_1 \}\label{q-list-d},
\end{align}
\end{subequations}
and, moreover,
\begin{equation}\label{q-monotonicity}
\left\langle Q_u(u), W_u(u) \right\rangle \geq 0, \text{ in } D\setminus \{ a_1 \},
\end{equation}
where we have set
\begin{equation}
D:= \mathrm{Int}\left( {\cup_{g\in G_{a_1}} g\overline{F}} \right).
\end{equation}
\end{hypothesis}

\begin{theorem}[\cite{alikakos-fusco}]\label{theorem1}
Under Hypotheses \ref{h1}--\ref{h4}, there exists an equivariant classical solution to system \eqref{system} such that
\begin{enumerate}
\item $|u(x)-a_1| \leq K \mathrm{e}^{-k d(x,\partial D)}$, for $x \in D$ and for positive constants $k$, $K,$ \medskip
\item $u(\overline{F}) \subset \overline{F}$ and $u(D) \subset D$.
\end{enumerate}
In particular, $u$ connects the $N=|G|/|G_{a_1}|$ global minima of $W$ in the sense that
\[ \lim_{\lambda \to +\infty} u(\lambda g \eta) = g a_1, \text{ for all } g \in G,\]
uniformly for $\eta$ in compact subsets of $D\cap\SF^{n-1}.$
\end{theorem}

\section{The extended Kato inequality}
We begin by presenting a straightforward extension of the classical Kato inequality. We follow the presentation in \cite[p.\ 85]{hislop-sigal}. Let $\hat{Q}: \R^m \to \R$ be a continuous function satisfying the following assumptions.
\begin{enumerate}
\item $\hat{Q}$ is convex,
\item $\hat{Q} > 0$ and $\hat{Q}_u \neq 0$, for $u \neq 0$,
\item $\hat{Q} = |u| + H(u)$, for a $C^2$ function $H: \R^m \to \R$, such that $H(0) = 0$ and $H_u(0) = 0$.
\end{enumerate}

\begin{lemma}\label{lemma1}
Let $u \in L^{\infty} (\R^n; \R^m)$ and suppose\footnote{The fact that $u$ should be in $u \in L^{\infty} (\R^n; \R^m)$ was pointed out to us by Panagiotis Smyrnelis. If $H$ is assumed globally Lipschitz, then $u \in L^{1}_{\mathrm{loc}} (\R^n; \R^m)$ suffices.} that the distributional Laplacian $\Delta u \in L^{1}_{\mathrm{loc}} (\R^n; \R^m)$. Then,
\begin{equation}\label{kato}
\Delta \hat{Q}(u) \geq \langle \Delta u , \hat{Q}_u (u) \rangle,
\end{equation}
in the distributional sense, with the definition
\[ \hat{Q}_u (u) := 
\begin{cases}
\nabla_{\! u} \hat{Q}(u), &\text{for } u \neq 0,\smallskip\\
0, &\text{for } u=0.
\end{cases} \]
\end{lemma}

\begin{remarks}
The well-known Kato inequality for functions $u \in L^{1}_{\mathrm{loc}} (\R^n; \mathbb{C})$ states that 
\begin{equation}
\Delta |u| \geq \mathrm{Re} [(\sgn u) \Delta u],
\end{equation}
in the distributional sense. The choice $|u| = \sqrt{u \bar{u}}$ and
\[ \sgn u = 
\begin{cases}
0, &\text{for } u=0,\smallskip\\
{\bar{u}}/{|u|}, &\text{for } u \neq 0,
\end{cases} \]
is a special case, for $\hat{Q}(u) = |u|$ and $\mathrm{Re} [u \bar{v}] = \langle u,v \rangle$. 

Also, under the hypothesis $u \in W^{1,2}_{\mathrm{loc}} (\R^n; \R^m) \cap L^{\infty} (\R^n; \R^m)$ we note that $\hat{Q}(u(\cdot)) \in W^{1,2}_{\mathrm{loc}} (\R^n; \R^m)$ (see \cite[p,\ 54]{kinderlehrer-stampacchia} or \cite[p.\ 130]{evans-gariepy}). Therefore, \eqref{kato} holds in $W^{1,2}_{\mathrm{loc}} (\R^n; \R^m)$.
\end{remarks}

\begin{proof}
We utilize the summation convention. We first establish
\begin{equation}\label{kato1}
\Delta \hat{Q} (u) \geq \langle \Delta u, \hat{Q}_u (u) \rangle,
\end{equation}
for $u \in C^{\infty}$, except where $\hat{Q} (u)$ is not differentiable. Set 
\[ \hat{Q}_{\varepsilon} (u(x)) = \sqrt{\hat{Q}^2 (u(x)) + \varepsilon^2}, \text{ for } \varepsilon > 0; \]
then,
\begin{equation}
\hat{Q}_{\varepsilon}\, \hat{Q}_{\varepsilon, i} = \hat{Q}\, \hat{Q}_{,u_k} u_{k,i},\text{ where } \hat{Q}_{\varepsilon, i} := \frac{\partial}{\partial x_i} \hat{Q}_{\varepsilon} (u(x)),
\end{equation}
and
\begin{equation}
(\hat{Q}_{\varepsilon, i})^2 = \left( \frac{\hat{Q}}{\hat{Q}_{\varepsilon}} \right)^2 (\hat{Q}_{,u_k} u_{k,i})^2 \leq (\hat{Q}_{,u_k} u_{k,i})^2.
\end{equation}
Hence,
\begin{equation}
\hat{Q}_{\varepsilon, i} \, \hat{Q}_{\varepsilon, i} \leq (\hat{Q}_{,u_k} u_{k,i}) (\hat{Q}_{,u_k} u_{k,i}),
\end{equation}
therefore
\begin{equation}\label{gradient1}
| \nabla_{\! x} \hat{Q}_{\varepsilon} |^2 \leq | (\nabla_{\! x} u)^\top \hat{Q}_u |^2.
\end{equation}
Moreover,
\begin{equation}
(\hat{Q}_{\varepsilon}\, \hat{Q}_{\varepsilon, i})_{,i} = \langle \Delta u, \hat{Q}\, \hat{Q}_u \rangle + \hat{Q} \langle (\partial^2 \hat{Q}) u_{,i}, u_{,i} \rangle + | (\nabla_{\! x} u)^\top \hat{Q}_u |^2,
\end{equation}
where $\partial^2 \hat{Q}$ is the Hessian of $\hat{Q}$ and $u_{,i} = (u_{1,i}, \dots ,u_{m,i})$. By convexity it follows that
\[ | \nabla_{\! x} \hat{Q}_{\varepsilon} |^2 + \hat{Q}_{\varepsilon}\, \Delta \hat{Q}_{\varepsilon} \geq \langle \Delta u, \hat{Q}\, \hat{Q}_u \rangle + | (\nabla_{\! x} u)^\top \hat{Q}_u |^2, \]
from which, by \eqref{gradient1}, 
\begin{equation}\label{laplacian1}
\Delta \hat{Q}_{\varepsilon} \geq \left\langle \Delta u, \frac{\hat{Q}\, \hat{Q}_u}{\hat{Q}_{\varepsilon}} \right\rangle.
\end{equation}
At points of smoothness we can take the limit $\varepsilon \to 0$ and obtain \eqref{kato1}.

We proceed by mollification. Let $w \in C^\infty (\R^n)$, with $w \geq0 $ and $\int w(x) \,\dd x = 1$. For $\delta > 0$ we define $w_\delta (x) = \delta^{-n} w(\delta^{-1} x)$ and set 
\[ I_\delta u := w_\delta * u, \text{ for } \delta > 0. \]
Then, $I_\delta u \to u$ and $\Delta (I_\delta u) \to \Delta u$ in $L^1$, as $\delta \to 0$. Applying \eqref{laplacian1} to $I_\delta u$ we have
\begin{equation}\label{fraction}
\Delta \hat{Q}_{\varepsilon}(I_\delta u) \geq \left\langle \Delta (I_\delta u), \frac{\frac{\partial}{\partial u}(\frac{1}{2}\hat{Q}^2)(I_\delta u)}{\hat{Q}_{\varepsilon}(I_\delta u)} \right\rangle.
\end{equation}
Taking $\delta \to 0$ and utilizing that $\hat{Q}^2$ is everywhere differentiable and that the fraction inside the inner product in \eqref{fraction} is bounded ($L^\infty$ requirement for $u(\cdot)$), by the dominated convergence theorem we have
\begin{equation}
\Delta \hat{Q}_{\varepsilon}(u) \geq \left\langle \Delta u, \frac{\frac{\partial}{\partial u}(\frac{1}{2}\hat{Q}^2)(u)}{\hat{Q}_{\varepsilon}(u)} \right\rangle.
\end{equation}
Finally, we pass to the limit in $\mathcal{D}^{\prime}$ as $\varepsilon \to 0$.
\end{proof}

\section{The gradient flow and positivity (\cite{alikakos-fusco})}
We define the set of {\em positive maps} (in the class of equivariant Sobolev maps)
\begin{equation}\label{pos}
\mathcal{U}^\mathrm{Pos}:=\big\{ u \in W^{1,2}_{\mathrm{E}}(B_R;\R^n) \mid u ( \overline{F_R} ) \subset \overline{F} \big\}
\end{equation}
and the set of {\em strongly positive maps}
\begin{equation}\label{strongly-pos}
\mathcal{U}^{\mathrm{Pos}}_{0}:=\big\{ u \in W^{1,2}_{\mathrm{E}}(B_R;\R^n) \mid u (F_R) \subset F \big\},
\end{equation}
where $F_R = F \cap B_R$. Here $R>0$ and clearly the sets $\mathcal{U}^\mathrm{Pos}$ and $\mathcal{U}^{\mathrm{Pos}}_{0}$ depend on $R$.

We will utilize the gradient flow
\begin{equation}\label{evolution-problem}
\begin{cases}
\dfrac{\partial u}{\partial t} = \Delta u - W_u(u), &\text{in } B_R \times (0,\infty),\bigskip\\
\dfrac{\partial u}{\partial \n} = 0, &\text{on } \partial B_R \times (0,\infty),,\bigskip\\
u(x,0) = u_0(x), &\text{in } B_R,
\end{cases}
\end{equation}
where ${\partial}/{\partial \n}$ is the normal derivative. We note that by Hypothesis \ref{h2}
\begin{equation}\label{evo1}
\langle -W_u(u), u \rangle \leq 0, \text{ for } |u| = M.
\end{equation}
We will consider initial conditions in \eqref{evolution-problem} satisfying in addition
\begin{equation}\label{evo2}
\| u_0 \|_{L^{\infty} (B_R;\R^n)} \leq M.
\end{equation}
Since $W$ is $C^2$ (cf.\ Hypothesis \ref{h1}), the results in \cite[Ch.\ 3, \S3.3, \S 3.5]{henry} apply and provide a unique solution to \eqref{evolution-problem} in $C(0,\infty; W^{1,2}_{\mathrm{E}}(B_R;\R^n))$, which for $t>0$, as a function of $x$, is in $C^{2+\alpha} (\overline{B_R};\R^n))$, for some $0<\alpha<1$. Moreover, the solution satisfies the estimate
\begin{equation}\label{evo3}
\| u(\cdot, t) \|_{L^{\infty} (B_R;\R^n)} \leq M, \text{ for } t \geq 0.
\end{equation}
This follows from \eqref{evo1}, \eqref{evo2}, and by well-known invariance results \cite[Ch.\ 14, \S B]{smoller}.

\begin{theorem}[\cite{alikakos-fusco}]\label{theorem2}
Let $W$ be a $C^2$ potential satisfying Hypothesis \ref{h2}. If $u_0 \in \mathcal{U}^{\mathrm{Pos}}$ and  $\| u_0 \|_{L^{\infty} (B_R;\R^n)} \leq M$, then
\[ u(\cdot, t; u_0) \in \mathcal{U}^{\mathrm{Pos}}, \text{ for } t \geq 0,  \]
and, moreover,
\[ u(\cdot, t; u_0) \in \mathcal{U}^{\mathrm{Pos}}_{0}, \text{ for } t > 0, \text{ provided } u_0 (\overline{F_R}) \cap F \neq \varnothing. \]
\end{theorem}
\begin{proof}
Let $u : B_R \to \R^n$ be an equivariant map. We will prove that $u$ is a positive map if and only if
\begin{equation}\label{closure1}
u(\overline{(\mathcal{P}_\gamma^+)_R} ) \subset\overline{\mathcal{P}_\gamma^+}, \text{ for all } \gamma\in\Gamma,
\end{equation}
where $(\mathcal{P}_\gamma^+)_R  = \mathcal{P}_\gamma^+\cap B_R$. 

Suppose that \eqref{closure1} holds. Then
\[ u(\overline{F_R}) = u ( \cap_{\gamma\in\Gamma} \overline{(\mathcal{P}_\gamma^+)_R} ) \subset \cap_{\gamma\in\Gamma}\, u (\overline{(\mathcal{P}_\gamma^+)_R}) \subset \cap_{\gamma\in\Gamma}\,\overline{\mathcal{P}_\gamma^+} = \overline{F}.\]
Hence, $u$ is positive. Conversely, suppose that $u$ is a positive equivariant map on $B_R$. Then, equivalently, $u_{\mathrm{e}}$ defined by
\begin{equation}\label{u-e-def}
u_{\mathrm{e}}(x) := 
\begin{cases}
u(x), &\text{for } x \in B_R\smallskip\\
0,  &\text{for } x \in \R^n \setminus B_R
\end{cases}
\end{equation}
is a positive equivariant map on $\R^n$. For any $g \in G$, we have from equivariance and positivity,
\begin{equation}\label{u-e}
u_{\mathrm{e}} (g(\overline{F})) = g(u_{\mathrm{e}}(\overline{F})) \subset g(\overline{F}).
\end{equation}
Now pick a $\gamma\in\Gamma$ and take an $x \in
 \mathcal{P}_\gamma^+ $ and fix it. There is a $g \in G$, denoted by
$g_x$, such that $x \in g_x(\overline{F})$ and $g_x(F)$ is also a
fundamental region. Since for each fundamenal region $F^\prime$ and for each reflection $\gamma$ we have either $F^\prime \subset \mathcal{P}_\gamma^+$ or $F^\prime \subset-\mathcal{P}_\gamma^+$, we conclude that
\begin{equation}
g_x(\overline{F}) \subset \overline{\mathcal{P}_\gamma^+}.
\end{equation}
Thus, by \eqref{u-e}, $u_{\mathrm{e}} (\overline{\mathcal{P}_\gamma^+}) \subset \overline{\mathcal{P}_\gamma^+}$, and so \eqref{closure1} follows.

Now consider \eqref{evolution-problem} with  $u_0 \in \mathcal{U}^\mathrm{Pos}$. By the regularizing property of the equation the solution is classical for $t >0$ and by \eqref{evo1} it exists globally in time and belongs to $C(0,+\infty; W^{1,2}_{\mathrm{E}}(B_R; \R^n)) \cap C^1(0, +\infty; C^{2+\alpha}(\overline{B_R}; \R^n)$, for some $0< \alpha <1$ (see \cite{henry}). Consider a reflection $\gamma\in\Gamma$ and set 
\begin{align*}
\zeta(x,t)& = \langle u(x,t,u_0),\eta_\gamma \rangle, \text{ on } B_{R} \times (0,\infty),\\
\zeta_0(x)& = \langle u_0(x),\eta_\gamma \rangle, \text{ on } B_{R}.
\end{align*}
By taking the inner product of equation \eqref{evolution-problem} with $\eta_\gamma$, we obtain
\begin{equation}\label{phi-problem}
\begin{cases}
\dfrac{\partial \zeta}{\partial t} = \Delta \zeta + c\zeta, &\text{in } B_{R} \times (0,\infty),\bigskip\\
\dfrac{\partial \zeta}{\partial \n} =0, &\text{on } \partial B_{R} \times (0,\infty),\bigskip\\
\zeta(\cdot,0) = \zeta_0,
\end{cases}
\end{equation}
where we have set 
\[ c(x,t)=\dfrac{\langle W_u(u(x,t,u_0),\eta_\gamma\rangle}{\zeta(x,t)}.\]
From  the equivariance of $u(\cdot,t,u_0)$ and $W_u(\gamma u) = \gamma W_u(u)$ it follows that 
\begin{align}
\zeta(x,t) &= -\zeta(\gamma x,t), \text{ in } B_{R} \times (0,\infty),\label{zeta-invariance} \\
c(x,t) &= c(\gamma x,t), \text{ in } B_{R} \times (0,\infty) \label{c-invariance}.
\end{align} 
From the symmetry of $W$ we also have that $u \in \pi_\gamma$ implies  $W_u(u) \in \pi_\gamma$. From this we deduce
\begin{equation}
\langle W_u(u),\eta_\gamma \rangle = \langle u, \eta_\gamma \rangle \left\langle
\int_0^1W_{uu}\big(u + (s-1)\langle u, \eta_\gamma \rangle
\eta_\gamma\big) \eta_\gamma \,\dd s,\, \eta_\gamma \right\rangle.
\end{equation}
Thus, the coefficient $c(x,t)$ of $\zeta$ in \eqref{phi-problem} is bounded (actually continuous) on $B_{R}\times(0,\infty)$. Since $u_0$ is a positive map, we have $\zeta_0 \geq 0$ for $\langle x,\eta_\gamma \rangle\geq 0$.   Therefore, for establishing positivity it is sufficient to show that $\zeta(x,t) \geq 0$, for $x\in B_R^+ = \{ x \in B_R \mid \langle x,\eta_\gamma \rangle> 0 \} $ and $t \geq 0$. We note that by \eqref{zeta-invariance} there holds $\zeta(x,t) = 0$ for $x\in \pi_\gamma \times [0,\infty)$, hence if $\zeta$ is a classical solution of \eqref{phi-problem}, we have that $\zeta(x,t)$ is nonnegative on $B_{R}^{+} \times [0,\infty)$ by the maximum principle. For general $\zeta_0 \in W^{1,2} (B_R)$ we approximate via mollification as in \cite[\S 4.2, Thm.\ 2]{evans-gariepy} and note that positivity and symmetry are preserved by the approximation process, rendering $\zeta_{0}^{\varepsilon} \in C^{\infty} ({B_R}) \cap L^\infty(B_R)$, with $\zeta_{0}^{\varepsilon} \to \zeta_0$ in $W^{1,2} (B_R)$. By the classical maximum principle, there holds that $\zeta^{\varepsilon} (x,t) \geq 0$ on $B_{R}^{+} \times [0,\infty)$, and by continuous dependence for \eqref{phi-problem} in $W^{1,2} (B_R)$ \cite[Thm.\ 3.4.1]{henry}, we have that $\zeta^\varepsilon (\cdot, t) \to \zeta (\cdot, t)$ a.e.\ in $B_R$ along subsequences $\varepsilon_n \to 0$, hence $\zeta(x,t) \geq 0$ a.e. Finally, since $\zeta(x,t)=0$ for $x \in \pi_\gamma \times (0,\infty)$ and since $\zeta(\cdot,t) \in C^{2+\alpha}(\overline{B_R})$ for $t>0$, the Hopf boundary lemma applies on the smooth part of $\partial B_{R}^{+}$ and renders
\[
\zeta(x,t) > 0, \text{ in } B_{R}^{+} \times (0,\infty),
\]
unless $\zeta(x,t) \equiv 0$, hence unless $\zeta_0(x) \equiv 0$. But the hypothesis $u_0(\overline{F_R}) \cap F \neq \varnothing$ excludes this second option.
\end{proof}

\section{The minimization}
Let $A^R := \left\{ u \in W^{1,2}(B_R,\R^n) \mid u(\overline{F}_R) \subset \overline{F} \right\}$ and consider the minimization problem 
\[
\min_{A^R} J_{B_R}, \text{ where } J_{B_R}(u) = \int_{B_R} \left\{ \frac{1}{2} |\nabla u|^2 + W(u) \right\} \dd x.
\]

We will argue first that the minimizer exists. We redefine $W(u)$ for $|u| \geq M+1$, so that the modified $W$ is $C^2$, satisfies $W(u) \geq c^2 |u|^2$, for $|u| \geq M+1$ and a constant $c$, and also $W(gu) = W(u)$, for all $g \in G$. We still denote the modified potential by $W$ and the modified functional by $J_{B_R}$. We note that the convexity  of $\overline{F}$ implies that $A^R$ is convex and closed in $W_{\mathrm{E}}^{1,2}(B_R;\R^n)$. The modified functional $J_{B_R}$ satisfies all the properties required by the direct method and, as a result, a minimizer $v_R \in A^R$ exists.

Next we will show that as a consequence of Hypothesis \ref{h2} we can produce a minimizer $u_R \in A^R$, which in addition satisfies the estimate $|u_R (x)| \leq M$ (cf.\ (H3) in \cite{bronsard-gui-schatzman}). Due to this estimate, the values of $W$ outside $\{ |u| \leq M \}$ will not matter in the considerations in the rest of the paper and, therefore, the equation that will be solved is \eqref{system} with the original unmodified potential $W$. Set 
\begin{equation}
u_R (x) = P v_R (x),
\end{equation}
where $P v$ equals the projection on the sphere $\{ v \in \R^n \mid |v| = M \}$, for points outside the sphere ($Pv = M v / |v|$), and equals the identity inside the sphere. Since $P$ is a contraction with respect to the Euclidean norm in $R^n$, it follows that $u_R \in W^{1,2} (B_R; \R^n)$, with $|\nabla u_R (x)| \leq |\nabla v_R (x)|$. Furthermore, 
\[ u_R(gx) = P v_R (gx) = P g v_R(x) = g Pv_R (x) = g u_R (x), \]
hence $u_R \in W_{\mathrm{E}}^{1,2} (B_R; \R^n)$. Clearly $u_R (\overline{F}) \subset \overline{F}$ and $|u_R (x)| \leq M$, for $x \in B_R$. 

The fact that $u_R$ is also a minimizer is a consequence of Hypothesis \ref{h2} and the following calculation.
\begin{align*}
J_{B_R} (u) &\geq \int_{B_R} \left\{ \frac{1}{2} |\nabla v_R|^2 + W(v_R) \right\} \dd x, \text{ for } u \in A^R,\\
&\geq \int_{B_R} \left\{ \frac{1}{2} |\nabla u_R|^2 + W(v_R) \right\} \dd x\\
&= \int_{|v_R (x)| \leq M} \left\{ \frac{1}{2} |\nabla u_R|^2 + W(u_R) \right\} \dd x\\ 
&\qquad + \int_{|v_R (x)| > M} \left\{ \frac{1}{2} |\nabla u_R|^2 + W(v_R) \right\} \dd x\\
&\geq \int_{|v_R (x)| \leq M} \left\{ \frac{1}{2} |\nabla u_R|^2 + W(u_R) \right\} \dd x\\ 
&\qquad + \int_{|v_R (x)| > M} \left\{ \frac{1}{2} |\nabla u_R|^2 + W(M \frac{v_R}{|v_R|}) \right\} \dd x\\
&= \int_{B_R} \left\{ \frac{1}{2} |\nabla u_R|^2 + W(u_R) \right\} \dd x,
\end{align*}
where the last inequality follows from Hypothesis \ref{h2}.

We will be constructing the solution by taking the limit 
\begin{equation}
u(x) = \lim_{R \to \infty} u_R (x).
\end{equation}
For this purpose, we will need to show that the positivity constraint built in $A^R$ does not affect the Euler--Lagrange equation, and we also need certain estimates, uniform in $R$, which in particular will imply that the solution is nontrivial.

\begin{lemma}\label{lemma3}
Let $u_R$ be as above. Then, for $R>1$, the following hold.
\begin{enumerate}
\item $J_{B_R} (u_R) \leq C R^{n-1}$, $\| u_R \|_{L^{\infty}(B_R; \R^n)} \leq M$, and $Q(u_R(x)) \leq \overline{Q}$, where $\overline{Q} := \max_{|u| \leq M} Q(u)$, 
\item $\Delta u_R - W_u (u_R) = 0$, in $W_{\mathrm{loc}}^{1,2} (B_R; \R^n)$,
\item $u_R$ is positive (cf.\ \eqref{pos}),
\item $\Delta Q(u_R (x)) \geq 0$, in $W_{\mathrm{loc}}^{1,2} (D_R)$, where $D_R := D \cap B_R$ (cf.\ Hypothesis \ref{h4}).
\end{enumerate}
\end{lemma}
\begin{proof}
For (i), define
\[
u_{\mathrm{aff}}(x) := 
\begin{cases}
d(x;\partial D) a_1, &\text{for } x \in D_{R} \text{ and } d(x;\partial D) \leq 1,\smallskip\\
a_1,  &\text{for } x \in D_{R} \text{ and } d(x;\partial D) \geq 1,
\end{cases} \]
and extend it equivariantly on $B_R$. Clearly, $u_{\mathrm{aff}} \in A^R$.
By the nonnegativity of $W$ and a simple calculation,
\begin{equation}
0 \leq J_{B_R} (u_R) \leq \min_{A^R} J_{B_R}(u) < J_{B_R} (u_\mathrm{aff}) < C R^{n-1},
\end{equation}
for some constant $C$ independent of $R$. The rest of (i) is already known.

For (ii), by Theorem \ref{theorem2}, we have $u(\cdot,t;u_R) \in A^R$, for $t\geq 0$. Since $u_R$ is a global minimizer of $J_{B_R}$ in $A^R$, and since $u(\cdot,t;u_R) \in C^1(0,\infty;C^{2+\alpha} (\overline{B_R}))$, a classical solution to \eqref{evolution-problem} for $t>0$, we conclude from
\begin{equation}\label{j-time}
\frac{\dd}{\dd t} J_{B_R} (u(\cdot,t)) = - \int_{B_R} |u_t|^2 \,\dd x
\end{equation} 
that $|u_t (x,t)|=0$, for all $x \in B_R$ and $t>0$. Hence, for $t>0$, $u(\cdot, t)$ is satisfying
\begin{equation}\label{u-satisfies}
\Delta u(x,t) - W_u (u(x,t)) = 0.
\end{equation}
By taking $t \to 0+$ and utilizing the continuity of the flow in $W^{1,2} (B_R;\R^n)$ at $t=0$, $u(\cdot,\cdot;u_R) \in C([0,\infty);W^{1,2}(B_R; \R^n))$, we obtain (ii).

Since (iii) is already known, we go on to (iv) where we obtain from \eqref{u-satisfies}, for $t>0$,
\begin{align*}
0 &= \langle Q_u(u(x,t)), \Delta u(x,t) \rangle - \langle Q_u(u(x,t)), W_u (u(x,t)) \rangle \\
&= \langle \hat{Q}_u (u(x,t) -a_1), \Delta( u(x,t) - a_1) \rangle - \langle Q_u(u(x,t)), W_u (u(x,t))
\end{align*}
where $Q(u) = \hat{Q}(u-a_1)$, while using \eqref{kato} we continue to obtain
\begin{align}
0 &\leq \Delta \hat{Q}(u(x,t) - a_1) - \langle Q_u(u(x,t)), W_u (u(x,t)) \rangle \nonumber\\
&= \Delta Q(u(x,t)) - \langle Q_u(u(x,t)), W_u (u(x,t)) \rangle \nonumber\\
&\leq \Delta Q(u(x,t)),
\end{align}
by Theorem \ref{theorem2}, utilizing $u_R \in \mathcal{U}^{\mathrm{Pos}}$, from which it follows that $u(\overline{D_R},t) \subset D$, and by Hypothesis \ref{h4}, particularly \eqref{q-monotonicity}.

Thus, by the second remark following Lemma \ref{lemma1}, we have
\begin{equation}
\Delta Q(u(x,t)) \geq 0, \text{ in } W^{1,2}_{\mathrm{loc}} (D_R), \text{ for } t>0,
\end{equation}
or,  equivalently,
\begin{equation}
\int_{D_R} \nabla Q(u(x,t)) \nabla \phi(x) \,\dd x \leq 0, \text{ for all } \phi \geq 0,\, \phi \in W^{1,2}_{\mathrm{loc}} (D_R).
\end{equation}
We will argue that 
\begin{equation}\label{nabla-conv}
\nabla Q(u(\cdot,t)) \to \nabla Q(u_R(\cdot)), \text{ weakly in } L^2(B_R), \text{ as } t \to 0,
\end{equation}
via which the proof of (iv) will be concluded. We know that 
\begin{equation}\label{u-conv}
\begin{cases}
u(\cdot,t;u_R) \to u_R, \text{in } W^{1,2}(B_R;\R^n), \text{ as } t \to 0,\smallskip\\
\| u(\cdot,t;u_R) \|_{L^\infty (B_R;\R^n)} \leq M.
\end{cases}
\end{equation}
Hence,
\[ Q(u(\cdot,t;u_R)) \to Q(u_R), \text{ in } L^2(B_R), \text{ as } t \to 0, \]
since $Q_u$ can be taken globally bounded. Thus,
\[ \nabla Q(u(\cdot,t;u_R)) \to \nabla Q(u_R), \text{ in } \mathcal{D}^{\prime}(B_R), \text{ as } t \to 0. \]
However, $\| \nabla Q(u(\cdot,t;u_R)) \|_{L^2(B_R)} < C$ by \eqref{u-conv}. Therefore \eqref{nabla-conv} is established and the proof is complete.
\end{proof}

The consideration in Lemma \ref{lemma3}, particularly \eqref{j-time}, together with the fact that $u_R$ is a global minimizer, show that $u(\cdot,t;u_R)$ is an equilibrium of \eqref{evolution-problem} for $t>0$, that is, a time-independent solution satisfying in addition the boundary condition ${\partial u}/{\partial \n}=0$. We can therefore replace $u_R$ with this equilibrium which satisfies all the properties of Lemma \ref{lemma3} and also is in $C^{2+\alpha} (\overline{B_R};\R^n)$.

\begin{corollary}\label{coro}
We may assume that $u_R \in C^{2+\alpha} (\overline{B_R};\R^n)$ is an equilibrium of \eqref{evolution-problem} that satisfies all the properties of Lemma \ref{lemma3}. Then,
\[ u_R(\overline{F_R}) \cap \overline{F} \neq \varnothing \quad\text{implies}\quad u_R \in \mathcal{U}^{\mathrm{Pos}}_{0}. \]
\end{corollary}
\begin{proof}
This follows from Theorem \ref{theorem2} and the fact that $u_R$ is a time-independent solution of \eqref{evolution-problem}.
\end{proof}

\section{The comparison function $\sigma$ (\cite{alikakos-fusco})}\label{section-comparison}
We prove three lemmas leading to the construction of a map $\sigma$ that will play a major role in the derivation of the uniform estimates in $R$ in the following section. We let $\chi_A$ be the characteristic function of a set $A$.
 
Given numbers $l, \lambda>0$, set $L=l+\lambda$ and let
$\varphi=\chi_{\overline{B_l}}\varphi_1+\chi_{\overline{B_L}\setminus\overline{B_l}}\varphi_2$,
where $\varphi_1:\overline{B_l}\to\R$, $\varphi_2:\overline{B_L}\setminus
B_l\to\R$ are defined by
\begin{equation}\label{final-comp}
\begin{cases}
\Delta \varphi_1 = c^2\varphi_1, &\text{in } B_l,\smallskip\\
\varphi_1 = \bar q, &\text{on } \partial B_l,
\end{cases}
\end{equation}
and
\begin{equation}\label{quasifinal-comp}
\begin{cases}
\Delta \varphi_2 = 0, &\text{in } B_L\setminus\overline{B_l},\smallskip\\
\varphi_2 = \bar q, &\text{on } \partial B_l,\smallskip\\
\varphi_2 = \overline{Q}, &\text{on } \partial B_L,
\end{cases}
\end{equation}
where $c$, $\bar q$, and $M$ below are the constants defined in Hypotheses \ref{h1} and \ref{h2} and
\begin{equation}\label{quasifinal-comp5}
\overline{Q} = \max_{\vert u\vert\leq M}Q(u),
\end{equation}
(see Hypothesis \ref{h2} and (i) of Lemma \ref{lemma3}). The map $\varphi$ is radial, that is, $\varphi_j(x)=\phi_j(\vert x \vert)$, for $j=1,2.$ Classical properties of Bessel functions imply that  $\phi_1 : [0,l] \to \R$ is positive and increasing together with the first derivative $\phi_1^\prime$. The function $\phi_2:[l,L]\to\R$ is increasing with decreasing first derivative $\phi_2^\prime$, by explicit calculation.

\begin{lemma}\label{lemma4}
The following hold.
\begin{enumerate}
\item The function $\phi_1^\prime(l)$ is strictly increasing for $l\in(0,+\infty)$ and
\begin{equation}\label{phi1-limit}
      \lim_{l\to+\infty}\phi_1^\prime(l) = c\bar q.
    \end{equation}
\item There exists a strictly increasing function $h:(0,+\infty)\to(0,+\infty)$ such that
\begin{equation}\label{ex-bound}
\phi_1(r) \leq \mathrm{e}^{h(l)(r-l)}\phi_1(l),\text{ for } r\in[0,l],
\end{equation}
and $\lim_{l\to+\infty} h(l) = c$.
\item There is a constant $C_0$, independent of $l$, such that
\begin{equation}
\phi_1^{\prime\prime}(r) \leq  C_0,\text{ for } r\in[0,l].
\end{equation}
\end{enumerate}
\end{lemma}

\begin{proof}
Statements (i) and (ii) are proved in \cite[Lemma 2.4]{flp}. For (iii) note that 
\begin{equation}
\phi_1^{\prime\prime} = c^2 \phi_1 - \frac{n-1}{r} \phi_1^{\prime} \leq c^2 \psi_1 \leq c^2 \bar{q},
\end{equation}
since $\phi_1$ is increasing and bounded by $\bar{q}$.
\end{proof}

An explicit computation yields, for $r\in[l,L],$
\begin{equation}\label{phi2-derivative}
\phi_2^\prime(r)=
\begin{cases}
\dfrac{\overline{Q}-\bar q}{r\log(L/l)}, &\text{for } n=2,\medskip\\
(n-2)\dfrac{l^{n-2}(\overline{Q}-\bar q)}{r^{n-1}(1-(l/L)^{n-2})}, &\text{for } n>2.
\end{cases}
\end{equation}

\begin{lemma}\label{lemma5}
The following hold.
\begin{enumerate}
\item Let the ratio $l/L$ be fixed. Then,
\begin{equation}\label{phi2-limit}
\lim_{l\to+\infty}\phi_2^\prime(l) = 0.
\end{equation}
\item Let the difference $L-l=\lambda$ be fixed. Then, $\phi_2^\prime(l)$ is a decreasing function of $l\in(0,+\infty)$ and
\begin{equation}\label{phi2-lim}
\lim_{l\to+\infty}\phi_2^\prime(r) = \frac{\overline{Q}-\bar q}{\lambda},\text{ for } r\in[l,l+\lambda].
\end{equation}
Moreover, there exists a constant $C_0$, independent of $l\in[1,+\infty)$, such that
\begin{equation}
\vert \phi_2^{\prime\prime}(r)\vert \leq \frac{C_0}{l},\text{ for } r\in[l,l+\lambda].
\end{equation}
\end{enumerate}
\end{lemma}

\begin{proof}
Statement (i) is a straightforward consequence of \eqref{phi2-derivative}. We prove (ii) for $n>2.$ The case $n=2$ is similar. To show that $\phi_2^\prime(l)$ is decreasing, we prove that the map $f(l) = l (1- (l / (l+\lambda))^{n-2})$ is increasing. Setting $\xi= l / (l+\lambda)$ we have
\[
f^\prime(l) = d(\xi) := 1-(n-1)\xi^{n-2}+(n-2)\xi^{n-1},\text{ for } \xi\in[0,1),
\]
and $f^\prime(l)>0$, for $l\in(0,+\infty)$, follows from $d(0)=1$, $d(1)=0$, and $d^\prime(\xi)<0$, for $\xi\in(0,1)$. The limit \eqref{phi2-lim} follows from \eqref{phi2-derivative}. The last statement of the lemma follows from
\[ 
\phi_2^{\prime\prime}(r)=-(n-1)\frac{l^{n-1}}{r^n}\phi_2^\prime(l). \qedhere
\]
\end{proof} 

Let $\varphi$ be as before and let $\delta>0$ be a small number. Denote by $\vartheta : B_{l+\delta}\setminus\overline{B_{l-\delta}} \to \R$ the solution of the problem
\begin{equation}\label{small-comp}
\begin{cases}
\Delta \vartheta = 0, &\text{in } B_{l+\delta}\setminus\overline{B_{l-\delta}},\smallskip\\
\vartheta = \varphi, &\text{on } \partial (B_{l+\delta}\setminus
\overline{B_{l-\delta}}).
\end{cases}
\end{equation}
We have $\vartheta(x)=\theta(\vert x\vert))$, where $\theta : [l-\delta, l+\delta] \to \R$ satisfies
\begin{equation}\label{theta-derivative}
\theta^\prime(r)=
\begin{cases}
\dfrac{\phi_2(l+\delta)-\phi_1(l-\delta)}{r\log\frac{l-\delta}{l-\delta}}, &\text{for } n=2,\bigskip\\ 
(n-2)\dfrac{(l-\delta)^{n-2}(\phi_2(l+\delta)-\phi_1(l-\delta))}{r^{n-1}(1-(\frac{l-\delta}{l+\delta})^{n-2})}, &\text{for } n>2.
\end{cases}
\end{equation}

\begin{lemma}\label{lemma6}
There exist positive constants $l_0,\, \lambda,\, \delta,\, \bar q^\prime < \bar{q},\, \delta^\prime,\, \mu$, such that $l\geq l_0$, $L=l+\lambda$ implies
\begin{enumerate}
\item $\phi_1^\prime(l)>\phi_2^\prime(l)+\mu,$
\item $\vartheta<\varphi, \text{ in } B_{l+\delta}\setminus\overline{B_{l-\delta}}$,
\item The map $\sigma: \overline{B_L}\to\R$ defined by $\sigma=\chi_{B_{l-\delta}\cup (\overline{B_L}\setminus \overline{B_{l+\delta}})}\varphi+\chi_{\overline{B_{l+\delta}}\setminus B_{l-\delta}}\vartheta$ satisfies
\begin{equation}\label{sigma-below}
\sigma\leq \bar q^\prime<\bar q, \text{ in } \overline{B_{l+\delta^\prime}}.
\end{equation}
\end{enumerate}
\end{lemma}

\begin{proof}
Letting the ratio $\rho= l/L$ be fixed, then \eqref{phi1-limit} and \eqref{phi2-limit} imply that there is an $l_0$ such that (i) holds for $l=l_0$ and some $\mu>0$. Fixing $\lambda=l_0 ((l / \rho)-1)$, then (i) holds for all $l\geq l_0$. This follows from Lemmas \ref{lemma4} and (ii) of Lemma \ref{lemma5}, which imply that $\phi_1^\prime(l)$ is increasing and $\phi_2^\prime(l)$ is decreasing for fixed $\lambda$. From \eqref{theta-derivative}, the relation
\[ 
\phi_2(l+\delta)-\phi_1(l-\delta)=(\phi_2^\prime(l)+\phi_1^\prime(l))\delta+o(\delta),
\]
which holds uniformly in $l$ since $\phi_1(l)=\phi_2(l)=\bar{q}$, and 
\[
\log\frac{l+\delta}{l-\delta} = 2\frac{\delta}{l}+o(\delta),\qquad
\left( \frac{l-\delta}{l+\delta} \right)^{n-2} = 1-2(n-2)\frac{\delta}{l}+o(\delta),
\]
it follows that
\begin{align}\label{theta-phi}
\left\vert \theta^\prime(r)-\frac{1}{2}(\phi_2^\prime(l)+\phi_1^\prime(l)) \right\vert &\leq C\delta, \text{ for } r\in[l-\delta,l+\delta],\\
\vert\theta^{\prime\prime}\vert &\leq \frac{C}{l}, \text{ for } r\in[l-\delta,l+\delta] 	
\end{align}
for some constant $C>0$, independent of $l\in[l_0,+\infty)$. From (i) and \eqref{theta-phi}, and the bounds on $\phi_1^{\prime\prime}$, $\phi_2^{\prime\prime}$, $\theta^{\prime\prime}$, it follows that there is a small $\delta>0$, independent of $l\in[l_0,+\infty)$, such that
\begin{equation*}\label{theta-phi2}
\begin{cases}
\theta^\prime(r) < \phi_1^\prime(r), &\text{for } r\in[l-\delta,l],\smallskip\\
\theta^\prime(r) > \phi_2^\prime(r), &\text{for } r\in[l,l+\delta].
\end{cases}
\end{equation*}
This and $\theta(l-\delta)=\phi_1(l-\delta)$, $\theta(l+\delta)=\phi_2(l+\delta)$, prove (ii). The existence of the number $\bar q^\prime < \bar q$ and $0 < \delta^\prime < \delta$, independent of $l\in[l_0,+\infty)$, follows by the same arguments and from the existence of the limits \eqref{phi1-limit} and \eqref{phi2-lim}.
\end{proof}

\section{Uniform estimates in $R$}
In this section we will make use of special notation. We denote by $B_R (x_R)$ the ball of radius $R>0$ centered at $x_R$. As before, $B_R$ denotes the ball of radius $R>0$ centered at the origin and $D_{4R} = D \cap B_{4R}$, with $x_R$ a point in $D_{4R}$ such that $B_R (x_R) \subset D_{4R}$. The function $u_{4R}$ is the minimizer for the functional $J_{B_{4R}}$ in Corollary \ref{coro}.

Set
\begin{equation}
v_R (x) := \frac{Q(u_{4R} (x)) - \bar{q}/2}{\overline{Q} - \bar{q}/2}, \text{ for } x \in B_R (x_R),
\end{equation}
where $\bar{q}$ as in Hypothesis \ref{h1}, $\overline{Q}$ as in Lemma \ref{lemma3}, Hypothesis \ref{h2}, with $\overline{Q} > \bar{q}/2$. We will also rescale the dependent variable via $y = (x - x_R) / R$ and define
\begin{equation}
\hat{v}_R (y) := v_R (Ry + x_R) = v_R (x), \text{ for } y \in \hat{B}_1,
\end{equation}
where $\hat{B}_{1} := \{ y \in \R^n \mid |y| < 1 \}$, $B_{1} := \{ x \in \R^n \mid |x| < 1 \}$, and ${B}_{R}^{+} (x_R) := \{ x \in {B}_{R} (x_R) \mid v_R (x) \geq 0 \}$, ${B}_{R}^{-} (x_R) := \{ x \in {B}_{R} (x_R) \mid v_R (x) \leq 0 \}$, and analogously, $\hat{B}_{1}^{+} := \{ y \in \hat{B}_{1} \mid \hat{v}_R (y) \geq 0 \}$, ${B}_{1}^{-} := \{ y \in \hat{B}_{1} \mid \hat{v}_R (y) \leq 0 \}$. Notice that $\hat{B}_{1}^{+}$, $\hat{B}_{1}^{-}$ depend on $R$.

By definition 
\begin{equation}
Q(u_{4R} (x)) \geq \frac{\bar{q}}{2}, \text{ on } B_{R}^{+} (x_R).
\end{equation}
By positivity ((iii) of Lemma \ref{lemma3}) and equivariance, there holds $u_{4R} (B_R(x_R)) \subset u_{4R} (D_{4R}) \subset \overline{D}$. Hence,
\begin{equation}\label{w4r}
W(u_{4R} (x)) \geq \varepsilon_0 (\bar{q}) > 0, \text{ on } B_{R}^{+} (x_R),
\end{equation}
since $a_1$ is the unique zero of $W$ in $\overline{D}$ (Hypotheses \ref{h3}, \ref{h4}).

\begin{lemma}\label{lemma8}
The following estimate holds for the Lebesgue measure of $\hat{B}_{1}^{-}$.
\begin{equation}\label{b-estimate}
|\hat{B}_{1}^{-}| \geq |\hat{B}_{1}| - \frac{C}{\varepsilon_0 (\bar{q}) R},
\end{equation}
where $C$ is a constant depending only on the constant $C$ in (i) of Lemma \ref{lemma3} and the dimension $n$.
\end{lemma}
\begin{proof}
We have
\begin{align*}
C R^{n-1} &\geq \int_{B_{4R}} W(u_{4R} (x)) \,\dd x \quad \text{(by (i) of Lemma \ref{lemma3})}\\
&\geq \int_{B_{R}^{+}(x_R)} W(u_{4R} (x)) \,\dd x \quad \text{(by } W \geq 0)\\
&\geq \varepsilon_0 (\bar{q})\, |B_{R}^{+} (x_R)| \quad \text{(by \eqref{w4r})}.
\end{align*}
Therefore,
\[ 
\frac{C}{R} \geq \varepsilon_0 (\bar{q}) \frac{|B_{R}^{+} (x_R)|}{R^n} = \varepsilon_0 (\bar{q})\, |\hat{B}_{1}^{+}|, \]
hence
\[ |\hat{B}_{1}^{-}| = |\hat{B}_{1}| - |\hat{B}_{1}^{+}| \geq |\hat{B}_{1}| - \frac{C}{\varepsilon_0 (\bar{q}) R}. \qedhere
\]
\end{proof}

\begin{remark}
The lemma above is a direct consequence of the basic integral estimate in (i) Lemma \ref{lemma3} and the assumption that in $D$ the potential has a unique zero. Estimate \eqref{b-estimate} states that the minimizer $u_{4R} (x)$ on a set of large measure in $B_R (x_R)$ is close to $a_1$, the zero of $W$, for $R \to \infty$. 
\end{remark}

The point in the next lemma is that the subharmonicity of $Q(u_{4R}(x))$ in $D$ (by (iv) of Lemma \ref{lemma3}) via a classical result of De Giorgi (see Appendix) allows us to obtain a pointwise estimate in the ball $B_{R/2} (x_R)$ of half the radius.

\begin{lemma}\label{lemma9}
Fix a $\delta \in (0,1)$. Then, for $R$ large enough such that
\begin{equation}\label{ed-estimate}
1 - \frac{1}{\varepsilon_0 (\bar{q})} \frac{C}{R} \frac{1}{c_0} \geq 1 - \delta,
\end{equation}
we have the estimate
\begin{equation}\label{sup-estimate}
\sup_{B_{R/2}(x_R)} Q(u_{4R} (x)) \leq \frac{\bar{q}}{2} + \mu (1-\delta) \left( \overline{Q} - \frac{\bar{q}}{2} \right),
\end{equation}
where $\mu(\cdot)$ is defined in the Appendix, with $\mu (1-\delta) < 1$.
\end{lemma}
Here $C$ is the constant in Lemma \ref{lemma8} and $c_0$ is the volume of the unit ball in $\R^n$.

\begin{proof}
Note that $\Delta_y \hat{v}_R \geq 0$ in $\hat{B}_1$ and $\hat{v}_R \leq 1$,  in $\hat{B}_1$, by (iv) and (i) of Lemma \ref{lemma3} respectively, and moreover
\[
\frac{|\hat{B}_{1}^{-}|}{|\hat{B}_{1}|} \geq 1 - \frac{1}{\varepsilon_0 (\bar{q})} \frac{C}{R} \frac{1}{c_0} \geq 1 - \delta
\]
by \eqref{b-estimate}. Hence, by the lemma in the Appendix,
\[
\sup_{\hat{B}_{1/2}} \hat{v}_R (y) \leq \mu (1-\delta) < 1,
\]
which is equivalent to \eqref{sup-estimate}.
\end{proof}

Next we will iterate. The number $\delta$ is fixed in Lemma \ref{lemma9} and we select $k$ as the minimal integer with the property
\begin{equation}\label{k-property}
\frac{\bar{q}}{2} + (\mu(1-\delta))^k \left( \overline{Q} - \frac{\bar{q}}{2} \right) < \bar{q}.
\end{equation}
Clearly $k$ depends only on $\delta$. Finally we choose $R_0 = R_0 (\delta)$ such that 
\begin{equation}\label{r-nought}
1 - \frac{1}{\varepsilon_0 (\bar{q})} \frac{C}{R} \frac{1}{|\hat{B}_{1/2^k}|} \geq 1-\delta, \text{ for } R \geq R_0,
\end{equation}
with $C$ as in Lemma \ref{lemma9}.

From now on, $R$, in the definition of $A^R$ and in the definition of the minimizer $u_R$, is assumed to satisfy \eqref{r-nought}, and free otherwise. For such an $R$ we define
\[
\begin{cases}
\overline{Q}_0 := \overline{Q},\medskip\\
\overline{Q}_i := \dfrac{\bar{q}}{2} + \mu (a^*) \left( \overline{Q}_{i-1} - \dfrac{\bar{q}}{2} \right), \medskip\\
v_i (x) := \dfrac{Q(u_{4R}(x)) - \bar{q}/2}{\overline{Q}_{i-1} - \bar{q}/2}, \text{ for } x \in B_{R/2^i} (x_R),\medskip\\
\hat{v}_i (y) := v_i (Ry), \text{ for } y \in \hat{B}_{1/2^i},
\end{cases}
\]
for $i=1,2,\ldots,k$ and $a^* = 1-\delta$. 

We notice that \eqref{r-nought} implies all the corresponding inequalities for $i=1,2,\ldots,k$ and, particular, \eqref{ed-estimate}.

\begin{lemma}
For an integer $k=k(\delta)$, as in \eqref{k-property}, and for $R \geq R_0 (\delta)$, as in \eqref{r-nought}, the following estimate holds.
\begin{equation}\label{lemma10-estimate}
\sup_{B_{R/2^k}(x_R)} Q(u_{4R} (x)) \leq \frac{\bar{q}}{2} + (\mu(a^*))^k \left( \overline{Q} - \frac{\bar{q}}{2}\right) < \bar{q}.
\end{equation}
\end{lemma}
\begin{proof}
We make the simple observation that
\begin{equation}\label{simple-obsrv}
\sup_{B_{R/2^i}(x_R)} Q(u_{4R} (x)) \leq \frac{\bar{q}}{2} + (\mu(a^*))^i \left( \overline{Q} - \frac{\bar{q}}{2} \right)
\end{equation}
holds for $i=1,2,\ldots,k$. 

We note that for $i=1$ this is just \eqref{sup-estimate}. Let us establish \eqref{simple-obsrv} for $i=2$. We may assume that $k\geq 3$ since otherwise we have the estimate we need, hence $\overline{Q}_1 > \bar{q}/2$. We have $Q(u_{4R}(x)) \leq \overline{Q}_1$, on $B_{R/2}(x_R)$, by \eqref{sup-estimate}, and $Q(u_{4R}(x)) \geq \bar{q}/2$, on $B_{R/2}^{+} (x_R)$, by definition. Hence,
\begin{align*}
C R^{n-1} &\geq \int_{B_{R}^{+}(x_R)} W(u_{4R} (x)) \,\dd x\\
&\geq \int_{B_{R/2}^{+}(x_R)} W(u_{4R} (x)) \,\dd x \quad \text{(cf.\ Proof of Lemma \ref{lemma8})}\\
&\geq \varepsilon_0 (\bar{q})\, |B_{R/2}^{+} (x_R)| \quad \text{(by \eqref{w4r})}.
\end{align*}
Therefore,
\[ 
\frac{C}{R} \geq \varepsilon_0 (\bar{q}) \frac{|B_{R/2}^{+} (x_R)|}{R^n} = \varepsilon_0 (\bar{q}) |\hat{B}_{1/2}^{+}|,
\]
hence,
\[
|\hat{B}_{1/2}^{-}| = |\hat{B}_{1/2}| - |\hat{B}_{1/2}^{+}| \geq |\hat{B}_{1/2}| - \frac{C}{\varepsilon_0 (\bar{q}) R}.
\]
It follows that
\begin{equation}
\frac{|\hat{B}_{1/2}^{-}|}{|\hat{B}_{1/2}|} \geq 1 - \frac{C}{\varepsilon_0 (\bar{q}) R |\hat{B}_{1/2}|} \geq 1 -\delta,
\end{equation}
by \eqref{r-nought}.

On the other hand, $\Delta_y \hat{v}_2 (y) \geq 0$ in $\hat{B}_{1/2}$ and $\hat{v}_2 \leq 1$, in $\hat{B}_{1/2}$, hence, by the lemma in the Appendix,
\[
\sup_{\hat{B}_{1/2^2}} \hat{v}_2 (y) \leq \mu(a^*),
\]
which equivalently gives
\[
\sup_{B_{R/2^2}(x_R)} Q(u_{4R}(x)) \leq \frac{\bar{q}}{2} + \mu(a^*) \left( \overline{Q}_1 - \frac{\bar{q}}{2} \right),
\]
or
\[
\sup_{B_{R/2^2}(x_R)} Q(u_{4R}(x)) \leq \frac{\bar{q}}{2} + \mu(a^*)^2 \left( \overline{Q} - \frac{\bar{q}}{2} \right).
\]
By repeating this process for $i=3,\ldots,k$, we obtain \eqref{lemma10-estimate}.
\end{proof}

So far we have established that 
\begin{equation}\label{so-far}
\sup_{{B}_{R^*}(x_R)} Q(u_{4R}(x)) \leq \bar{q},
\end{equation}
where $R^* = {R} / {2^k}$, for $R \geq R_0$, and an integer $k$ independent of $R$. Utilizing the comparison function $\sigma$ in Section \ref{section-comparison} it is possible to show that the ball $B_{R^*} (x_R)$ in the supremum in \eqref{so-far} can be replaced by a large set $D_{R}^{*}$ which includes all of $D_{4R}$ with the exception of a strip along the boundary $\partial D$ of width $d_0$ independent of $R$, for $R \geq R_0$, that is,
\begin{equation}\label{drstar-supset}
D_{R}^{*} \supset \{ x \in D_{4R} \mid d(x,\partial D_{4R}) \geq d_0 \},
\end{equation}
for some $d_0 >0$, which depends on $l_0$ in Lemma \ref{lemma6}.

\begin{lemma}\label{lemma11}
The following estimate holds.
\begin{equation}\label{lemma11-estimate}
\sup_{D_{R}^{*}} Q(u_{4R}) \leq \bar{q},
\end{equation}
where $D_{R}^{*}$ has the properties stated above.
\end{lemma}
\begin{proof}
First we note that by Hypotheses \ref{h1}, \ref{h4},
\[ 
\langle Q_u(u), W_u(u) \rangle \geq c^2 Q(u), \text{ for } |u-a_1| \leq \bar{q},
\]
which implies, via Lemma \ref{lemma1}, (ii) of Lemma \ref{lemma3}, and \eqref{so-far}, the estimate
\begin{equation}\label{laplace-q-estimate}
\Delta Q (u_{4R}) \geq \langle \Delta u_{4R}, Q_u(u_{4R}) \rangle = \langle W_u (u_{4R}), Q_u(u_{4R}) \rangle \geq c^2 Q(u_{4R}),
\end{equation}
in $W^{1,2}_{\mathrm{loc}} (B_{R^*} (x_R))$.

Next we refer to Section \ref{section-comparison}. Consider a ball $B_l (\xi)$, tangent to $\partial B_{R^*} (x_R)$ and with its center $\xi$ inside $B_{R^*} (x_R)$, and also consider the concentric ball $B_L (\xi)$. Notice that $B_l (\xi)$ is the translation of $B_l$ and $B_L (\xi)$ the translation $B_L$. Similarly consider the translations of $\varphi_1$, $\varphi_2$, $\vartheta$, which we still denote by the same symbols. 

We now observe by \eqref{final-comp}, \eqref{so-far}, and \eqref{laplace-q-estimate}, that
\[
\begin{cases}
\Delta \varphi_1 = c^2 \varphi_1, &\text{in } B_l (\xi),\smallskip\\
\varphi_1 = \bar{q}, &\text{on } \partial B_l (\xi),
\end{cases}
\qquad
\begin{cases}
\Delta Q(u_{4R}) \geq c^2 Q(u_{4R}), &\text{in } B_l (\xi),\smallskip\\
Q(u_{4R}) \leq \bar{q}, &\text{on } \partial B_l (\xi),
\end{cases}
\]
hence, by the maximum principle for $W^{1,2}$ solutions (see \cite{gilbarg-trudinger}), we have 
\begin{equation}
Q(u_{4R}) \leq \varphi_1, \text{ in } B_l (\xi).
\end{equation}
Also, by \eqref{quasifinal-comp}, (i) and (iv) of Lemma \ref{lemma3}, and \eqref{so-far},
\[
\begin{cases}
\Delta \varphi_2 = 0, &\text{in } B_L (\xi) \setminus \overline{B_l (\xi)},\smallskip\\
\varphi_2 = \bar{q}, &\text{on } \partial B_l (\xi),\smallskip\\
\varphi_2 = \overline{Q}, &\text{on } \partial B_L (\xi), 
\end{cases}
\qquad
\begin{cases}
\Delta Q(u_{4R}) \geq 0, &\text{in } B_L (\xi),\smallskip\\
Q(u_{4R}) \leq \varphi_2, &\text{on } \partial( B_L (\xi) \setminus \overline{B_l (\xi)}),
\end{cases}
\]
hence,
\begin{equation}
Q(u_{4R}) \leq \varphi_2, \text{ in } B_L (\xi) \setminus \overline{B_l (\xi)}.
\end{equation}
We deduce therefore by Lemma \ref{lemma6} that
\begin{equation}\label{q-deduce}
\begin{cases}
Q(u_{4R}) \leq \varphi, &\text{in } B_L (\xi),\smallskip\\
Q(u_{4R}) \leq \vartheta, &\text{in } B_{l+\delta} (\xi) \setminus \overline{B_{l-\delta} (\xi)},\smallskip\\
Q(u_{4R}) \leq \sigma \leq \bar{q}^{\prime} < \bar{q}, &\text{in } \overline{B_{l+\delta^{\prime}} (\xi)}
\end{cases}
\end{equation}
Thus, we see from (iii) of \eqref{q-deduce} that the estimate \eqref{so-far} holds on a set larger than $B_{R^*} (x_R)$. Clearly, by repeating this process we obtain \eqref{lemma11-estimate}.
\end{proof}

We are now able to finish the proof of the theorem.
\begin{proof}[Proof of Theorem \ref{theorem1}]
First note that if $q : D_{R}^{*} \to \R$ is the solution to 
\begin{equation}\label{solutionto}
\Delta q = c^2 q, \text{ in } D_{R}^{*},\smallskip\\
q = \bar{q}^{\prime}, \text{ on } D_{R}^{*},
\end{equation}
then,
\begin{equation}\label{exp-estimate}
q(x) \leq K \mathrm{e}^{-k d(x, \partial D_{R}^{*})}, \text{ for } x \in D_{R}^{*}, 
\end{equation}
for positive constants $K$, $k$, independent of $R$.

Indeed, by the maximum principle, there holds $q \leq \bar{q}^{\prime}$. It follows that if $\varphi$ is the solution of equation \eqref{solutionto} on the ball with center $x$ and radius $d(x, \partial D_{R}^{*})$ and with boundary condition $\varphi = \bar{q}^{\prime}$, we have $q \leq \varphi$. This and the estimate \eqref{ex-bound} in (ii) of Lemma \ref{lemma4} imply \eqref{exp-estimate}. Moreover, we note that if $d_0 >0$ is as in \eqref{drstar-supset}, then 
\begin{equation}\label{exp-estimate2}
q(x) \leq K \mathrm{e}^{-k d(x, \partial D_{4R})}, \text{ in } B_{d_0} (x) \subset D_{4R},
\end{equation}
since $d(x, \partial D_{4R}) \leq d(x, \partial D_{R}^{*}) + d_0$. This last inequality follows from \eqref{drstar-supset} with a new constant $K$.

Now utilizing \eqref{so-far} and $\Delta Q(u_{4R}) \geq c^2 Q(u_{4R})$, in $W^{1,2}_{\mathrm{loc}}(D_{R}^{*})$, by \eqref{laplace-q-estimate}, we obtain by comparing with \eqref{solutionto} that $Q(u_{4R}(x)) \leq q(x)$, for $x \in D_{R}^{*}$, and, therefore, by \eqref{exp-estimate2} and \eqref{q-list-c},
\begin{equation}\label{uni-bound}
|u_{4R} (x) - a_1| \leq K \mathrm{e}^{-k d(x, \partial D_{4R})}.
\end{equation} 
The uniform bound in (i) of Lemma \ref{lemma3} and elliptic regularity, via a diagonal argument, allow us to pass to the limit along a subsequence in $R$ and capture a function
\[
u(x) = \lim_{R_{n} \to \infty} u_{R_n} (x).
\]
The uniform bound \eqref{uni-bound} implies that the limit function satisfies the exponential estimate in the theorem and is also a solution to 
\[
\Delta u -W_u(u) = 0, \text { in } \R^n,
\]
by (ii) of Lemma \ref{lemma3} (and the comment after its proof). Clearly, also $u \in \mathcal{U}^{\mathrm{Pos}}$. 

Finally we argue the strong positivity for $u(x)$ with respect to $D$. Take now an open connected set $U \subset D$ which contains some points far enough from $\partial D$ so that by the exponential estimate there holds $u(U) \cap D \neq \varnothing$. Define $\Gamma^\prime := \{ \gamma \in \Gamma \mid \pi_\gamma \cap D \neq \varnothing \}$; then, $D = \cap_{\gamma \in \Gamma \setminus \Gamma^\prime} \mathcal{P}_{\gamma}^{+}$ (see Lemmas 2.1 and 5.1  in \cite{alikakos-fusco}). As in the second part  of the proof of Theorem \ref{theorem2}, particularly \eqref{phi-problem}, set $z(x) = \langle u(x), \eta_\gamma \rangle$, for $x \in U$ and $\gamma \in \Gamma \setminus \Gamma^\prime$. Then,
\[
\begin{cases}
\Delta z + cz = 0, &\text{in } U,\smallskip\\
z \geq 0, &\text{in } U. \quad \text{(by positivity)}
\end{cases}
\]
By a well-known variant of the strong maximum principle, there holds $z > 0$ in $U$, unless $z \equiv 0$. Triviality is excluded by $u(U) \cap D \neq \varnothing$ above. From this, strong positivity follows.

The proof is complete.
\end{proof}

\section*{Appendix}\label{appendix}
We state a special case of the {\em De Giorgi oscillation lemma} which was originally established for general elliptic operators $\mathcal{L} u := \dv (A(x)\nabla u)$, with bounded, measurable coefficients (see \cite[p.\ 195]{caffarelli-salsa}).
\begin{de-lemma}[De Giorgi]\label{app-lemma}
Consider the ball $\hat{B}_1 = \{ y \in \R^n \mid |y| \leq 1 \}$ and a function $\hat{v} = \hat{v}(y)$ which satisfies
\begin{enumerate}
\item $\Delta_y \hat{v} \geq 0$, in $W^{1,2}_{\mathrm{loc}} (\hat{B}_1)$,
\item $\hat{v} \leq 1$, in $\hat{B}_1$,
\item ${|\hat{B}_{1}^{-}|} / {|\hat{B}_1|} \geq a^* > 0$, for $\hat{B}_{1}^{-} = \{ y \in \hat{B}_{1} \mid \hat{v}(y) \leq 0 \}$,
\end{enumerate}
where $|\Omega|$ is the Lebesgue measure of the set $\Omega$. Then, 
\[ \sup_{\hat{B}_{1/2}} \hat{v} \leq \mu (a^*) < 1. \]
\end{de-lemma}

Notice that the statement of the lemma is invariant under the scaling $y \to \lambda y$, for $\lambda > 0$, hence $\hat{B}_1$ can be replaced by $\hat{B}_\lambda$ and $\hat{B}_{1/2}$ by $\hat{B}_{\lambda /2}$, without affecting $\mu (a^*)$.

\section*{Acknowledgments}
Thanks are due to Giorgio Fusco for discussions on the contents of the paper, to Panagiotis Smyrnelis for pointing out two errors in a former version of the manuscript, and to Apostolos Damialis for his suggestions for improving the presentation.

\nocite{*}
\bibliographystyle{plain}

\end{document}